\documentclass[11pt, oneside]{amsart}
\usepackage{fullpage}
\usepackage{cite}
\usepackage[utf8]{inputenc}
\usepackage[T1]{fontenc}
\usepackage[english]{babel}
\usepackage{amsmath,amsfonts,amsthm, mathrsfs,amssymb}
\usepackage[linktocpage=true]{hyperref}
\usepackage{tikz-cd}
\usepackage{array}
\usepackage{ulem}

\theoremstyle{plain}
\newtheorem{theorem}{Theorem}[section]
\newtheorem{corollary}[theorem]{Corollary}
\newtheorem{lemma}[theorem]{Lemma}
\newtheorem{proposition}[theorem]{Proposition}

\theoremstyle{definition}

\newtheorem{remark}[theorem]{Remark}
\numberwithin{equation}{section}
\numberwithin{figure}{section}
\numberwithin{table}{section}

\usetikzlibrary{shapes.multipart}
\usetikzlibrary{patterns}
\usetikzlibrary{shapes.multipart}
\usetikzlibrary{arrows}
\usetikzlibrary{decorations.markings}
\usepgflibrary{decorations.shapes}
\usetikzlibrary{decorations.shapes}
\usepgflibrary{shapes.symbols}
\usetikzlibrary{shapes.symbols}
\usetikzlibrary{decorations.pathreplacing}

\tikzstyle{fleche}=[>=stealth', postaction={decorate}, thick]
\tikzstyle{axis}=[->, >=stealth', thick, gray]
\tikzstyle{path}=[->, >=stealth', thick]
\tikzstyle{grille}=[dotted, gray]
\colorlet{gray}{white!85!black}
\colorlet{red}{white!30!red}
\colorlet{blue}{white!75!blue}
\hypersetup{colorlinks=true, linkcolor=black}
\begin{document}
\title{\Large Skew odd orthogonal characters and interpolating Schur polynomials}
\author{Naihuan Jing}
\address{Department of Mathematics, North Carolina State University, Raleigh, NC 27695, USA}
\email{jing@ncsu.edu}
\author{Zhijun Li$^{\dagger}$ }
\address{School of Science, Huzhou University, Huzhou, Zhejiang 313000, China}
\email{zhijun1010@163.com}
\author{Danxia Wang}
\address{School of Science, Huzhou University, Huzhou, Zhejiang 313000, China}
\email{dxwangmath@126.com}
\author{Chang Ye, \MakeLowercase{and appendix with} Xinyu Pan}
\address{School of Science, Huzhou University, Huzhou, Zhejiang 313000, China}
\email{yechang@zjhu.edu.cn}
\address{School of Science, Huzhou University, Huzhou, Zhejiang 313000, China}
\email{pxypanxinyu@163.com}
\thanks{{\scriptsize
\hskip -0.6 true cm MSC (2020): Primary: 05E05; Secondary: 17B37.
\newline $^\dag$ Corresponding author: zhijun1010@163.com
}}
\maketitle
\begin{abstract}
We introduce two vertex operators to realize skew odd orthogonal characters $so_{\lambda/\mu}(x^{\pm})$ and derive
the Cauchy identity for the skew characters via Toeplitz-Hankel-type determinant similar to the Schur functions.
The method also gives new proofs of the Jacobi--Trudi identity and Gelfand--Tsetlin patterns for $so_{\lambda/\mu}(x^{\pm})$.
Moreover, combining the vertex operators related to characters of types $C,D$ (\cite{Ba1996,JN2015}) and the new vertex operators related to $B$-type characters, we obtain three families of symmetric polynomials that interpolate among characters of $SO_{2n+1}(\mathbb{C})$, $SO_{2n}(\mathbb{C})$
and $Sp_{2n}(\mathbb{C})$,
Their transition formulas are also explicitly given among symplectic and/or orthogonal characters and odd orthogonal characters.
\end{abstract}
\section{introduction}
It is well known that the characters of irreducible representations of the general linear group $GL_n(\mathbb{C})$ are the Schur functions $s_\lambda(x)$ indexed by partitions $\lambda=(\lambda_1,\dots,\lambda_n)$ and form an orthonormal basis of the ring of symmetric functions in the variables $x :=(x_1,\dots,x_n)$ \cite{Mac1995}. One of us \cite{Jing1991} used the vertex operator realization of the Schur functions to study symmetric functions in infinitely many variables, and the same method was extended to other types in \cite{JN2015}. The skew Schur functions $s_{\lambda/\mu}(x)$ indexed by skew partitions $\lambda/\mu$ are a fundamental family of symmetric functions that can be used to study restriction
to subrepresentations. To be precise, the skew Schur functions obey the general branching rule \cite[(5.9)]{Mac1995}
\begin{equation}
s_\lambda(x_1,\dots,x_n)=\sum_{\mu\subset\lambda}s_\mu(x_1,\dots,x_{n-k})s_{\lambda/\mu}(x_{n-k+1},\dots,x_n),
\end{equation}
where $\mu=(\mu_1,\dots,\mu_{n-k})$ are partitions and skew shapes $\lambda/\mu$ subject to $\mu_i\leq \lambda_i$.

Irreducible characters of the classical groups ($Sp_{2n}(\mathbb{C})$, $SO_{2n}(\mathbb{C})$ and $SO_{2n+1}(\mathbb{C})$) are known as the symplectic characters $sp_\lambda(x^{\pm})$, (even) orthogonal characters $o_\lambda(x^{\pm})$ and odd orthogonal characters $so_\lambda(x^{\pm})$ when expressed
as symmetric Laurent functions in $x^{\pm} :=(x_1,x^{-1}_1\dots,x_n,x^{-1}_{n})$, where $\lambda=(\lambda_1,\dots,\lambda_n)$ are the partitions indexing the representations \cite{FH1991}. All aforementioned functions admit determinantal expressions, such as Jacobi--Trudi identities and dual Jacobi--Trudi identities. Baker \cite{Ba1996} and Jing-Nie \cite{JN2015} constructed the vertex operator realizations for the symplectic and orthogonal characters. In \cite{JLW2024}, the authors used the vertex algebraic method to derive the combinatorial structures (Jacobi--Trudi formulas, Gelfand--Tsetlin patterns and Cauchy identities) for the skew symplectic characters $sp_{\lambda/\mu}(x^{\pm})$ and skew (even) orthogonal characters $o_{\lambda/\mu}(x^{\pm})$ (or type $D$), and showed they obey the following general branching rules
\begin{align}\label{e:JT-1}
sp_\lambda(x^{\pm};y^{\pm})&= \sum_{\mu=(\mu_1,\dots,\mu_{n-k})}sp_\mu(x^{\pm})sp_{\lambda/\mu}(y^{\pm}),\\ \label{e:JT-2}
o_\lambda(x^{\pm};y^{\pm})&= \sum_{\mu=(\mu_1,\dots,\mu_{n-k})}o_\mu(x^{\pm})o_{\lambda/\mu}(y^{\pm}),
\end{align}
where $x^{\pm}=(x_1,x^{-1}_1\dots,x_{n-k},x^{-1}_{n-k})$ and $y^{\pm}=(y_1,y^{-1}_1\dots,y_{k},y^{-1}_{k})$.
More recently Albion, Fischer, H\"ongesberg, and Schreier-Aigner \cite[(3a),(3c)]{AFHS2023} have given combinatorial proofs for the dual Jacobi--Trudi formulas associated with our formulas corresponding to \eqref{e:JT-1}-\eqref{e:JT-2} in terms of the elementary symmetric functions, in addition they also obtained the dual Jacobi-Trudi
formulas for the orthogonal characters of odd case.

In this paper, we introduce two new vertex operators for the {\it odd orthogonal characters} $so_\lambda(x^{\pm})$ (or type $B$) and show that they
can be used to investigate the remaining `odd' case--vertex operator realizations for the irreducible characters of the classical group $\mathrm{SO}_{2n+1}(\mathbb{C})$. We will show that the generated skew symmetric functions match with Abion et al's formula \cite[4(b)]{AFHS2023}. Using the same methodology in \cite{JLW2024}, we obtain the combinatorial identities for skew odd orthogonal characters $so_{\lambda/\mu}(x^{\pm})$ as well as the general branching rules
\begin{eqnarray}
so_\lambda(x^{\pm};y^{\pm})=\sum_{\mu=(\mu_1,\dots,\mu_{n-k})}so_\mu(x^{\pm})so_{\lambda/\mu}(y^{\pm}).
\end{eqnarray}
We also find three Toeplitz or Hankel-type determinants related to Schur functions (see Theorem \ref{e:th1}).

Recently Bisi and Zygouras \cite{BZ2022} have found $CB$ (resp. $DB$)-interpolating Schur polynomials $s^{CB}_\lambda(x;\alpha)$ (resp. $s^{DB}_\lambda(x;\alpha)$), which specialize to characters of type $C$ and $B$ (resp. $D$ and $B$), and obtained unified identities and transition formulas between the characters of the classical groups. We will construct three families of interpolating Schur polynomials $s^{BD}_\lambda(x;\alpha)$, $s^{BC}_\lambda(x;\alpha)$ and $s^{CD}_\lambda(x;\alpha)$ by vertex operators and show that
\begin{align*}
&s^{BD}_\lambda(x;0):=so_\lambda(x^{\pm}),\qquad\qquad\qquad s^{BD}_\lambda(x;1):=o_\lambda(x^{\pm}),\\
&s^{BC}_\lambda(x;0):=so_\lambda(x^{\pm}),\qquad\qquad\qquad s^{BC}_\lambda(x;-1):=sp_\lambda(x^{\pm}),\\
&s^{CD}_\lambda(x;0):=sp_\lambda(x^{\pm}),\qquad\qquad\qquad s^{CD}_\lambda(x;1):=o_\lambda(x^{\pm}).
\end{align*}
Along the way we also derive unified identities and transition formulas among the irreducible characters of the classical groups.

This paper is organized as follows. In Section 2, we introduce the vertex operators for the odd orthogonal characters.
In section 3, we show that the vertex operator realization derives the Jacobi--Trudi identity, Gelfand--Tsetlin patterns and Cauchy-type identity for skew odd orthogonal characters. In particular, some Toeplitz-Hankel-type determinants related to Schur functions are also given. In Section 4, we
obtain three families of interpolating Schur polynomials and discuss the transition
formulas among symplectic characters, orthogonal characters and odd orthogonal characters. In the Appendix, we show how to get dual Jacobi--Trudi identity for skew odd orthogonal characters through vertex operators.

\section{Preliminaries}
Let $\mathcal{H}$ be the Heisenberg algebra generated by $\{a_n|n\neq 0\}$ with the central element $c=1$ subject to the commutation relations\cite{FK1980}
\begin{align}\label{e:he1}
[a_m,a_n]=m\delta_{m,-n}c,~~~~[a_n,c]=0.
\end{align}
The Fock space $\mathcal{M}$ (resp. $\mathcal{M}^*$) is generated by the vacuum vector $|0\rangle$ (resp. dual vacuum vector $\langle0|$) and
subject to
\begin{align*}
a_n|0\rangle=\langle0|a_{-n},~~n>0.
\end{align*}

 We define the following vertex operators
\begin{equation}
\begin{aligned}\label{e:so}
U(z)&=(1+z)\exp\left(\sum^\infty_{n=1}\frac{a_{-n}}{n}z^n\right)\exp\left(-\sum^\infty_{n=1}\frac{a_n}{n}(z^{n}+z^{-n})\right)=\sum_{n\in \mathbb{Z}}U_nz^{-n},\\
U^*(z)&=(1-z)\exp\left(-\sum^\infty_{n=1}\frac{a_{-n}}{n}z^n\right)\exp\left(\sum^\infty_{n=1}\frac{a_n}{n}(z^{n}+z^{-n})\right)=\sum_{n\in \mathbb{Z}}U^*_nz^{n}.
\end{aligned}
\end{equation}
It is easy to check that
\begin{align}
\label{e:ac10}&U_n|0\rangle=U^*_{-n}|0\rangle=0,~~n>0,\\
\label{e:ac1}&\langle 0|U_n=\langle 0|U_{-n-1},~~\langle 0|U^*_n=-\langle 0|U^*_{-n+1}.
\end{align}

One can obtain generalized Clifford
algebra with the help of the Baker-Campbell-Hausdorff (BCH) formula
\begin{align*}
\exp(A)\exp(B)=\exp(B)\exp(A)\exp([A,B]),~~~~~~~~~~~[[A,B],A]=[[A,B],B]=0.
\end{align*}
\begin{proposition}
Operators $U_i,U^*_i$ satisfy the commutation relations
\begin{align}
\label{e:com5}U_iU_j+U_{j+1}U_{i-1}=0,\qquad U^*_iU^*_j+U^*_{j-1}U^*_{i+1}=0,\qquad U_iU^*_j+U^*_{j+1}U_{i+1}=\delta_{i,j}.
\end{align}
\end{proposition}
\begin{proof} Since the three relations are proved similarly, we consider the last equation and let
\begin{align*}
:U(z)U^*(w):&=(1+z)(1-w)\exp\left(\sum^\infty_{n=1}\frac{a_{-n}}{n}(z^n-w^n)\right)\exp\left(-\sum^\infty_{n=1}\frac{a_n}{n}(z^{n}+z^{-n}-w^{n}-w^{-n})\right).
\end{align*}
By the BCH formula,
\begin{align*}
U(z)U^*(w)&=\frac{1}{(1-\frac{w}{z})(1-zw)}:U(z)U^*(w):,\\
U^*(w)U(z)&=\frac{1}{(1-\frac{z}{w})(1-zw)}:U(z)U^*(w):.
\end{align*}
Thus we have
\begin{align*}
\frac{1}{z}U(z)U^*(w)+\frac{1}{w}U^*(w)U(z)&=\delta(z-w),
\end{align*}
where $\delta(z-w)=\sum_{n\in \mathbb{Z}}z^{-n-1}w^n$. Taking coefficients one can finish the proof.
\end{proof}
A {\it generalized partition} $\lambda=(\lambda_1,\lambda_2,\dots,\lambda_l)$ of weight $|\lambda|=\sum_i\lambda_i$
is a set of weakly decreasing nonnegative integers\footnote{Recall that a partition is a set of weakly decreasing positive integers. Thus a partition is a special case of generalized partitions.}. Nonzero $\lambda_i$ are called the parts of $\lambda$, and the length of $\lambda$ is the number of parts, denoted by $l(\lambda)$. The conjugate partition $\lambda^{\prime}$ is $(\lambda^{\prime}_1,\dots,\lambda^{\prime}_m)$ where $\lambda^{\prime}_i$ is the number of $j$ such that $\lambda_j\geq i$. A generalized partition $\lambda=(\lambda_1,\dots,\lambda_{l})$ is said to {\it interlace} the generalized partition $\mu=(\mu_1,\dots,\mu_{l-1})$, written as $\mu\prec\lambda$, if $\lambda_{i+1}\leq \mu_i\leq\lambda_i$ for $1\leq i\leq l-1$.
For a generalized partition $\lambda=(\lambda_1,\lambda_2,\dots,\lambda_l)$, let
\begin{align}
|\lambda^{so}\rangle=U_{-\lambda_1}U_{-\lambda_2}\cdots U_{-\lambda_l}|0\rangle, \qquad\langle \lambda^{so}|=\langle 0|U^*_{-\lambda_l}\cdots U^*_{-\lambda_1}.
\end{align}
Let $\widetilde{\mathcal{M}^*}$ be the completion of $\mathcal{M}^*$ defined by
 the filtration $\widetilde{\mathcal{M}^*}^{N}$=span$\{\langle 0|U^*_{-\nu_N}\cdots U^*_{-\nu_1}| \nu_1\geq \dots\geq \nu_N\geq 0, \nu_i\in \mathbb Z\}$.

Note that $\langle \lambda^{so}|$ belongs to $\widetilde{\mathcal{M}^*}$ rather than $\mathcal{M}^*$. For example, $\langle 0|U^*_{0}U^*_{-1}U^*_{-4}\neq \langle 0|U^*_{-1}U^*_{-4}$ due to $\langle 0|U^*_0\neq \langle 0|$. Therefore the partition element $\langle 0|U^*_{0}U^*_{-1}U^*_{-4}$ and $\langle 0|U^*_{-1}U^*_{-4}$ are different in $\widetilde{\mathcal{M}^*}$.

Note that the condition \eqref{e:ac10} (resp. \eqref{e:com5}) agrees with \cite[(2.3)]{JLW2024} (resp. \cite[Proposition 3.6]{JN2015} or \cite[(2.6)]{JLW2024}).
Using the same process in \cite[Th. 2.3]{JLW2024}, we have the following orthogonality relation.
\begin{proposition}\label{pro5}
For two generalized partitions $\mu=(\mu_1,\dots,\mu_l)$ and $\lambda=(\lambda_1,\dots,\lambda_l)$, one has that
\begin{align}\label{e:ortho1}
\langle \mu^{so}|\lambda^{so}\rangle=\delta_{\lambda\mu},
\end{align}
where $\langle u|v\rangle$ denotes $\langle u||v\rangle$. Here some parts of $\lambda$ and $\mu$ can be zeros and $\langle 0|1|0\rangle=1$.
\end{proposition}
We remark that though $|\lambda^{so}\rangle$ only depends on its nonzero parts, it has only one dual element in
the completed dual $\widetilde{\mathcal{M}^*}$ with fixed length.
 For example, the dual element of $U_{-4}U_{-1}|0\rangle$ in $\widetilde{\mathcal{M}^*}^{2}$ is $\langle 0|U^*_{-1}U^*_{-4}$, while in $\widetilde{\mathcal{M}^*}^{4}$ is $\langle 0|U^*_{0}U^*_{0}U^*_{-1}U^*_{-4}$.
\begin{corollary}For a generalized partition $\mu=(\mu_1,\dots,\mu_l)$,
\begin{align}\label{e:identity}
\langle \mu^{so}|\sum_{\eta=(\eta_1,\dots,\eta_l)}|\eta^{so}\rangle\langle \eta^{so}|=\langle \mu^{so}|,~~~~\sum_{\eta=(\eta_1,\dots,\eta_l)}|\eta^{so}\rangle\langle \eta^{so}||\mu^{so}\rangle\langle=|\mu^{so}\rangle,
\end{align}
where the sum is over all generalized partitions $\eta=(\eta_1,\dots,\eta_l)$.
\end{corollary}
We use the notation $(^a_b)$ to mean either $a$ or $b$. For a generalized partition $\mu=(\mu_1,\dots, \mu_l)$, one can get the following from \eqref{e:ac1} and \eqref{e:com5}
\begin{align*}
\langle 0|\left(^{U^*_{-\mu_l}}_{-U^*_{\mu_l+1}}\right)\cdots \left(^{U^*_{-\mu_i}}_{-U^*_{\mu_i+2l+1-2i}}\right)\cdots \left(^{U^*_{-\mu_1}}_{-U^*_{\mu_1+2l-1}}\right)=\langle\mu^{so}|.
\end{align*}
Note that there is a similar action of $S_l$ on the vectors $\langle 0|U^*_{-n_l}\cdots U^*_{-n_i}\cdots U^*_{-n_1}$
using the commutation relation $U^*_iU^*_j=-U^*_{j-1}U^*_{i+1}$ \eqref{e:com5}. Therefore we have that
\begin{align}\label{e:so9}
\varepsilon(\sigma)\langle 0|\left(^{U^*_{-\mu_{\sigma(l)}+\sigma(l)-l}}_{-U^*_{\mu_{\sigma(l)}-\sigma(l)+l+1}}\right)\cdots \left(^{U^*_{-\mu_{\sigma(i)}+\sigma(i)-i}}_{-U^*_{\mu_{\sigma(i)}-\sigma(i)+2l+1-i}}\right)\cdots \left(^{U^*_{-\mu_{\sigma(1)}+\sigma(1)-1}}_{-U^*_{\mu_{\sigma(1)}-\sigma(1)+2l}}\right)=\langle\mu^{so}|.
\end{align}
Similarly, by \eqref{e:ac10} and the commutation relation $U_iU_j=-U_{j+1}U_{i-1}$ \eqref{e:com5}, one also has
\begin{align}\label{e:permu1}
\varepsilon(\sigma) U_{-\mu_{\sigma_1}+\sigma_1-1}\cdots U_{-\mu_{\sigma_i}+\sigma_i-i}\cdots U_{-\mu_{\sigma_l}+\sigma_l-l}|0\rangle=|\mu^{so}\rangle.
\end{align}

\section{skew odd orthogonal characters} We extend the procedure in \cite{JLW2024} of employing  the vertex operator presentations of the odd orthogonal characters and their skew-types to derive the combinatorial expressions such as the Jacobi--Trudi identity, Gelfand--Tsetlin patterns and Cauchy-type identity for the skew odd versions.

The main idea is to express the skew orthogonal characters as matrix coefficients of certain half vertex operators. For the skew Schur polynomials
the half vertex operators come from the Bernstein vertex operators \cite{JLW2024} (see also \cite{Oko2001}). In this section we introduce
two types of parametric half vertex operators (\eqref{e:so11} and \eqref{e:so31}--\eqref{e:so30} respectively)
for the (skew) odd orthogonal characters and the interpolating Schur polynomials \cite{BZ2022}.

We first define the following half vertex operators
\begin{align}\label{e:eq10}
\Gamma_+(z)&=\exp\left(\sum^\infty_{n=1}\frac{a_n}{n}z^n\right),\\
\label{e:half2}\Gamma_-(z)&=\exp\left(\sum^\infty_{n=1}\frac{a_{-n}}{n}z^n\right),\\
\label{e:cd2}\widetilde{\Gamma}_+(z)&=\exp\left(-\sum^\infty_{n=1}\frac{a_{2n}}{n}z^n\right).
\end{align}
Next for the alphabets $x=(x_1,\dots,x_N)$, $x^{\pm}=(x_1,x^{-1}_1,\dots,x_N,x^{-1}_N)$ and the parameter $\alpha$, we introduce the following operators:
\begin{align}
\Gamma_+(x)&=\prod^{N}_{i=1}\Gamma_+(x_i),\\
\label{e:so11}\Gamma_+(x^{\pm})&=\prod^{N}_{i=1}\Gamma_+(x_i)\Gamma_+(x^{-1}_i),\\
\label{e:half1}\Gamma_-(x)&=\prod^{N}_{i=1}\Gamma_-(x_i),\\
\label{e:so31}\Gamma_+(x^{\pm};\alpha)&=\Gamma_+(x^{\pm})\Gamma_+(\alpha),\\
\label{e:so32}\bar{\Gamma}_+(x^{\pm};\alpha)&=\Gamma_+(x^{\pm})\Gamma^{-1}_+(\alpha),\\
\label{e:so30}\widetilde{\Gamma}_+(x^{\pm};\alpha)&=\Gamma_+(x^{\pm})\widetilde{\Gamma}_+(\alpha).
\end{align}
\subsection{Jacobi--Trudi identity} The odd orthogonal characters $so_\nu(x^{\pm})$, the irreducible characters for the odd orthogonal group $SO(2N+1,\mathbb{C})$, are given by the Weyl formula \cite[(24.28)]{FH1991}\cite[P. 38]{Me2019}
\begin{align}\label{e:so-character}
so_\nu(x^{\pm})=\frac{\det(x^{\nu_j+(N-j+\frac{1}{2})}_i-x^{-\nu_j-(N-j+\frac{1}{2})}_i)^N_{i,j=1}}{\det(x^{N-j+\frac{1}{2}}_i-x^{-(N-j+\frac{1}{2})}_i)^N_{i,j=1}}.
\end{align}
Using the same method in \cite[Proposition 3.3]{JLW2024}, we
derive the Jacobi--Trudi identity for the skew odd orthogonal characters as follows.
\begin{proposition}\label{pro6}
For the generalized partitions $\mu=(\mu_1,\dots,\mu_l)$ and $\lambda=(\lambda_1,\dots,\lambda_{l+N})$, one has that
\begin{align}\label{e:so22}
\langle\mu^{so}|\Gamma_+(x^{\pm})|\lambda^{so}\rangle=so_{\lambda/\mu}(x^{\pm})=\det(a_{ij})_{1\leq i,j\leq l+N}
\end{align}
where
\begin{equation}\label{e:so1}
a_{ij}=\left\{\begin{aligned}
&h_{\lambda_i-\mu_j-i+j}(x^{\pm})~~&~~1\leq j\leq l,\\
&h_{\lambda_i-i+j}(x^{\pm})+h_{\lambda_i-i-j+2l+1}(x^{\pm})~&~~l+1\leq j\leq l+N.
\end{aligned}
\right.
\end{equation}
\end{proposition}

\begin{proof} It follows from the definition of the vertex operator $U^*(z)$ that
\begin{align}\label{e:eq5}
\notag\langle 0|\Gamma_+(x^{\pm})=&\prod_{1\leq i\leq N}\frac{1}{1-x_i}\prod_{1\leq i< j\leq N}\frac{1}{1-x_ix_j}\prod_{1\leq i<j\leq N}\frac{1}{1-\frac{x_i}{x_j}}\langle 0|U^*(x_N)\cdots U^*(x_1)\\
=&\frac{\prod^{N}_{i=1}x^{i-\frac{1}{2}-N}_i}{\det(x^{-(N-j+\frac{1}{2})}_i-x^{N-j+\frac{1}{2}}_i)^N_{i,j=1}}\langle 0|U^*(x_N)\cdots U^*(x_1),
\end{align}
where we have used the Vandermonde type identity \cite[(4.6)]{JN2015}:
\begin{align}\label{e:va1}
\det(x^{j-1}_i-x^{2N-j}_i)^N_{i,j=1}=\prod_{1\leq i\leq N}(1-x_i)\prod_{1\leq i<j\leq N}(x_j-x_i)\prod_{1\leq i< j\leq N}(1-x_ix_j).
\end{align}
In view of \eqref{e:so9} for partition $\nu=(\nu_1,\dots,\nu_N)$, the coefficient of $\langle\nu^{so}|$ in $\langle 0|U^*(x_N)\cdots U^*(x_1)$ is
\begin{align}\sum_{\sigma\in S_N}\varepsilon(\sigma)\prod^{N}_{i=1}(x^{-\nu_{\sigma(i)}+\sigma(i)-i}_i-x^{\nu_{\sigma(i)}-\sigma(i)+2N+1-i}_i)
=\det(x^{-\nu_j+j-i}_i-x^{\nu_j-j+2N+1-i}_i)^N_{i,j=1}.
\end{align}
Consequently
\begin{align}\label{e:so12a}
\notag\langle 0|\Gamma_+(x^{\pm})
=&\sum_{\nu=(\nu_1,\dots,\nu_N)}\langle\nu^{so}|\frac{\prod^{N}_{i=1}x^{i-\frac{1}{2}-N}_i\det(x^{-\nu_j+j-i}_i-x^{\nu_j-j+2N+1-i}_i)^N_{i,j=1}}{\det(x^{-(N-j+1)}_i-x^{N-j+1}_i)^N_{i,j=1}}\\
\notag=&\sum_{\nu=(\nu_1,\dots,\nu_N)}\langle\nu^{so}|\frac{\det(x^{-\nu_j-(N-j+\frac{1}{2})}_i-x^{\nu_j+(N-j+\frac{1}{2})}_i)^N_{i,j=1}}{\det(x^{-(N-j+\frac{1}{2})}_i-x^{N-j+\frac{1}{2}}_i)^N_{i,j=1}}\\
\notag=&\sum_{\nu=(\nu_1,\dots,\nu_N)}\langle\nu^{so}|\frac{\det(x^{\nu_j+(N-j+\frac{1}{2})}_i-x^{-\nu_j-(N-j+\frac{1}{2})}_i)^N_{i,j=1}}{\det(x^{N-j+\frac{1}{2}}_i-x^{-(N-j+\frac{1}{2})}_i)^N_{i,j=1}}\\
=&\sum_{\nu=(\nu_1,\dots,\nu_N)}\langle\nu^{so}|so_\nu(x^{\pm}),
\end{align}
where we have recalled the bialternant formula of the odd orthogonal character $so_\nu(x^{\pm})$ associated to the partition $\nu$ \eqref{e:so-character}.
Invoking \eqref{e:so9} again and also by the Jacobi--Trudi formula for the odd orthogonal characters \cite[Proposition 24.33]{FH1991}
\begin{align*}
so_\nu(x^{\pm})=\det\left(h_{\nu_i-i+j}(x^{\pm})+h_{\nu_{i}-i-j+1}(x^{\pm})\right)_{1\leq i,j\leq N},
\end{align*}
where $h_n(x^{\pm})$ \cite{Ba1996,JN2015,Wey1946} is defined by \footnote{The generating function corresponds to $\phi_{sp}^{-1}$ in \cite{KT1987}.}
\begin{equation}
\prod^N_{i=1}\frac{1}{(1-x_iz)(1-x^{-1}_iz)}=\sum_{n\in \mathbb{Z}}h_n(x^{\pm})z^n,
\end{equation}
where $h_n(x^{\pm})=0$ for $n<0$.
We obtain another expression for $\langle 0|\Gamma_+(x^{\pm})$:
\begin{align}\label{e:so13}
\langle 0|\Gamma_+(x^{\pm})=\sum_{n_1\geq -N+1,n_2\geq -N+2,\dots,n_N\geq 0}
\langle 0|U^*_{-n_N}\cdots U^*_{-n_1}\prod^N_{i=1}(h_{n_i}(x^{\pm})+h_{n_i-2i+1}(x^{\pm})).
\end{align}
Using the vertex operator $U^*(z)$, we can write that
\begin{align}\label{e:com4}
\notag U^*(z)\Gamma_+(x^{\pm})=&\prod^N_{i=1}\frac{1}{(1-x_iz)(1-x^{-1}_iz)}\Gamma_+(x^{\pm})U^*(z)\\
=&\sum_{i\geq 0}h_i(x^{\pm})z^i\Gamma_+(x^{\pm})U^*(z),
\end{align}
which implies
\begin{align*}
U^*_{-n}\Gamma_+(x^{\pm})=\Gamma_+(x^{\pm})\sum_{i\geq 0}h_i(x^{\pm})U^*_{-n-i}.
\end{align*}

It follows from \eqref{e:so13} and \eqref{e:so9} that
\begin{equation}\label{e:so14}
\begin{aligned}
\langle\mu^{so}|\Gamma_+(x^{\pm})
=&\sum_{i_1,\dots,i_l\geq 0}h_{i_1}(x^{\pm})\cdots h_{i_l}(x^{\pm})\langle 0|\Gamma_+(x^{\pm})U^*_{-\mu_l-i_l}\cdots U^*_{-\mu_1-i_1}\\
=&\sum_{\substack{i_1,\dots,i_l\geq 0\\i_{l+1}\geq -N+1,\dots,i_{l+N}\geq 0}}\langle 0|U^*_{-i_{l+N}}\cdots U^*_{-i_{l+1}}U^*_{-\mu_l-i_l}\cdots U^*_{-\mu_1-i_1}\\
 &\qquad\qquad \cdot\prod^{l}_{k=1}h_{i_k}(x^{\pm})\prod^{l+N}_{k=l+1}\left(h_{i_k}(x^{\pm})+h_{i_k-2(k-l)+1}(x^{\pm})\right)\\
=&\sum_{\lambda=(\lambda_1,\dots,\lambda_{l+N})}\langle\lambda^{so}|\sum_{\sigma\in S_{l+N}}\prod^{l}_{k=1}h_{\lambda_{\sigma(k)}-\mu_k-\sigma(k)+k}(x^{\pm})\\
 &\qquad\qquad\cdot\prod^{l+N}_{k=l+1}\left(h_{\lambda_{\sigma(k)}-\sigma(k)+k}(x^{\pm})+h_{\lambda_{\sigma(k)}-\sigma(k)-k+2l+1}(x^{\pm})\right)\\
=&\sum_{\lambda=(\lambda_1,\dots,\lambda_{l+N})}\langle\lambda^{so}|\det(a_{ij})_{1\leq i,j\leq l+N},
\end{aligned}
\end{equation}
where $a_{ij}$ was defined in \eqref{e:so1} and in the last two equations $\lambda_{\sigma(i)}-\sigma(i)+2(l+N)+1-i$ is greater than $l+N-i$.
\end{proof}

From \eqref{e:identity} and the proof of Proposition \ref{pro6}, we know that for $\lambda=(\lambda_1,\dots,\lambda_{l+N})$
\begin{equation}
\begin{aligned}\label{e:bran1}
so_\lambda(y^{\pm};x^{\pm})=\langle 0|\Gamma_+(y^{\pm};x^{\pm})|\lambda^{so}\rangle=&\langle 0|\Gamma_+(y^{\pm})\sum_{\mu=(\mu_1,\dots,\mu_l)}|\mu^{so}\rangle\langle \mu^{so}|\Gamma_+(x^{\pm})|\lambda^{so}\rangle\\
=&\sum_{\mu=(\mu_1,\dots,\mu_l)}so_\mu(y^{\pm})so_{\lambda/\mu}(x^{\pm}),
\end{aligned}
\end{equation}
where $y^{\pm}=(y^{\pm}_1,\dots,y^{\pm}_l),~x^{\pm}=(x^{\pm}_1,\dots,x^{\pm}_N)$.
\begin{remark}Equation \eqref{e:bran1} can be viewed as the general branching rule for the odd orthogonal characters which justify
why $so_{\lambda/\mu}(x^{\pm})$ is called the skew odd orthogonal character (cf. \cite{KT1987}).
\end{remark}
\begin{remark} The skew Jacobi--Trudi formulas for the symplectic group $Sp_{2n}$ and the orthogonal group $SO_{2n}$
were derived by the authors in \cite{JLW2024}. Albion, Fischer, H\"ongesberg, and Schreier-Aigner then used lattice paths
to prove these skew Jacobi--Trudi formulas and also the one for $SO_{2n+1}$ \cite[(4b)]{AFHS2023}, which agrees with the right equation
in \eqref{e:so22}.
\end{remark}
\subsection{Gelfand--Tsetlin patterns} For two partitions $\gamma\subset\eta=(\eta_1,\dots,\eta_{k+1})$, the skew Schur function \cite[p.72]{Mac1995} at two variables $t_1, t_2$
can be written as
\begin{align}\label{e:skew-S1}
s_{\eta/\gamma}(t_1,t_2)=\det\left(h_{\eta_i-\gamma_j-i+j}(t_1,t_2)\right)_{1\leq i,j\leq k+1}=\sum_{\eta\prec\beta\prec\gamma}t^{|\beta|-|\eta|}_1t^{|\gamma|-|\beta|}_2.
\end{align}

\begin{lemma}
For partitions $\nu=(\nu_{1},\dots,\nu_{l})$ and $\lambda=(\lambda_{1},\dots,\lambda_{l+1})$, one has
\begin{align}
so_{\lambda/\nu}(t^{\pm})=\sum_{\nu\prec\alpha\prec\lambda}t^{2|\alpha|-|\lambda|-|\nu|}
\end{align}
summed over all partitions $\alpha$ such that $\nu\prec\alpha\prec\lambda$ and $\alpha_{l+1}\in \frac{1}{2}\mathbb{Z}$ with $0\leq \alpha_{l+1}\leq \min\{\lambda_{l+1},\nu_{l}\}$\footnote{Using exactly the same method in \cite[Lemma 4.4]{JLW2024}, one can also prove the result.}.
\end{lemma}
\begin{proof} Proposition \ref{pro6} implies that
\begin{align}\label{e:GT5}
\notag so_{\lambda/\nu}(t^{\pm})=&\langle\nu^{so}|\Gamma_+(t^{\pm})|\lambda^{so}\rangle\\
=&\det\left(h_{\lambda_i-\nu_j-i+j}(t^{\pm})\right)_{1\leq i,j\leq l+1}+\det\left(h_{(\lambda_i-1)-\vartheta_j-i+j}(t^{\pm})\right)_{1\leq i,j\leq l+1},
\end{align}
where $\vartheta=(\nu_1-1,\dots,\nu_l-1,0)$ and $\lambda-(1)^{l+1}=(\lambda_1-1,\dots,\lambda_{l+1}-1)$. It is easy to see that if $\nu_l=0$, the last two columns of the second determinant in \eqref{e:GT5} are equivalent, i.e., the determinant is nonzero only for $\vartheta$ being a partition. Thus
\begin{align}\label{e:soGT1}
so_{\lambda/\nu}(t^{\pm})=s_{\lambda/\nu}(t^{\pm})+s_{\lambda-(1)^{l+1}/\vartheta}(t^{\pm}).
\end{align}
By \eqref{e:skew-S1}, we have
\begin{align}
\notag s_{\lambda/\nu}(t,t^{-1})&=\sum_{\nu\prec\alpha\prec\lambda}t^{2|\alpha|-|\lambda|-|\nu|},\\
\label{e:skew-S2}s_{\lambda-(1)^{l+1}/\vartheta}(t^{\pm})&=\sum_{\vartheta\prec\beta\prec\lambda-(1)^{l+1}}t^{2|\beta|-|\lambda|-|\nu|+2l+1}.
\end{align}
Let $\pi=(\beta_1+1,\dots,\beta_l+1,\beta_{l+1}+\frac{1}{2})$. Then we have $\nu\prec\pi\prec\lambda$ with $\pi_{l+1}\in \mathbb{Z}+\frac{1}{2}$ and $\frac{1}{2}\leq\pi_{l+1}\leq\min\{\lambda_{l+1}-\frac{1}{2},\nu_{l}-\frac{1}{2}\}$. Since $\pi_{l+1}$ is a half integer, the interval could also be $0\leq\pi_{l+1}\leq\min\{\lambda_{l+1},\nu_{l}\}$. Thus we can rewrite \eqref{e:skew-S2} by
\begin{align*}
s_{\lambda-(1)^{l+1}/\vartheta}(t^{\pm})&=\sum_{\nu\prec\pi\prec\lambda}t^{2|\pi|-|\lambda|-|\nu|}.
\end{align*}
Then by \eqref{e:soGT1}, we can finish the proof.
\end{proof}
\begin{theorem}
For partitions $\mu=(\mu_1,\dots,\mu_l)\subset \lambda=(\lambda_1,\dots,\lambda_{l+N})$, one has
\begin{align}
\label{e:so19}&so_{\lambda/\mu}(x^{\pm})=\sum_{\mu=z_0\prec z_1\prec\dots\prec z_{2N}=\lambda}\prod^N_{i=1}x^{2|z_{2i-1}|-|z_{2i}|-|z_{2i-2}|}_i
\end{align}
where partitions $z_k=(z_{k,1},\dots,z_{k,l+\lceil\frac{ k}{2}\rceil})$ satisfy the odd orthogonal Gelfand--Tsetlin pattern
\begin{align*}
z_{k+1,j}\leq z_{k,j-1}\leq z_{k+1,j-1}~~~~\text{for}~~~~ 2\leq j\leq l+\lceil\frac{ k+1}{2}\rceil
\end{align*}
subject to $z_{2i-1,l+i}\in \frac{1}{2}\mathbb{Z}$ and $0\leq z_{2i-1,l+i}\leq \min\{z_{2i,l+i},z_{2i-2,l+i-1}\}$.
\end{theorem}
\begin{remark} The Gelfand--Tsetlin relation \eqref{e:so19} for the skew odd orthogonal character agrees with \cite[Definition 6.14]{AF2020}.
\end{remark}
\subsection{Cauchy-type identity}From the definitions of vertex operator $U(z)$ \eqref{e:so} and half vertex operator $\Gamma_-(z)$ \eqref{e:half2}, one has
 \begin{align}
 \Gamma_-(z)=\frac{1}{1+z}U(z)\Gamma_+(z)\Gamma_+(z^{-1}),
 \end{align}
 where half vertex operator $\Gamma_+(z)$ is defined by \eqref{e:eq10}.
By half vertex operator $\Gamma_-(\{y\})$ \eqref{e:half1} and the BCH formula, we have that\footnote{One can follow the proof of \cite[Lemma 5.1]{JLW2024} for more details.}
\begin{align}
\notag\Gamma_-(y)|0\rangle&=\prod_{1\leq i\leq N}(1+y_i)^{-1}\prod_{1\leq i<j\leq N}(1-y_iy_j)^{-1}(1-\frac{y_j}{y_i})^{-1}U(y_1)\cdots U(y_N)|0\rangle\\
\label{e:so5}&=\prod_{1\leq i\leq N}(1+y_i)^{-1}\prod_{1\leq i<j\leq N}(1-y_iy_j)^{-1}\sum_{l(\lambda)\leq N}s_\lambda(y)|\lambda^{so}\rangle,
\end{align}
where $y=(y_1,y_2,\dots,y_N)$, and we used relation \eqref{e:permu1} and the bialternant definition for Schur polynomials
\begin{align*}
s_\lambda(y)=\frac{\det\left(y^{\lambda_j-j+i}_i\right)_{1\leq i,j\leq N}}{\prod_{1\leq i<j\leq N}(1-\frac{y_j}{y_i})}=\frac{\det\left(y^{\lambda_j+N-j}_i\right)_{1\leq i,j\leq N}}{\det\left(x^{N-j}_i\right)_{1\leq i,j\leq N}}.
\end{align*}

For $x^{\pm}=(x^\pm_1,\dots,x^\pm_N)$, using the BCH formula, one has
\begin{align}
\label{e:so6}\langle0|\Gamma_+(x^\pm)\Gamma_-(y)|0\rangle=\prod^N_{i,j=1}(1-x_iy_j)^{-1}(1-x^{-1}_iy_j)^{-1}.
\end{align}

Propositions \ref{pro5} \& \ref{pro6} and \eqref{e:so5} also give
\begin{align}
\label{e:so7}\langle0|\Gamma_+(x^\pm)\Gamma_-(y)|0\rangle=\prod_{1\leq i\leq N}(1+y_i)^{-1}\prod_{1\leq i<j\leq N}(1-y_iy_j)^{-1}\sum_{l(\lambda)\leq N}so_\lambda(x^{\pm})s_\lambda(y).
\end{align}

Therefore we have that
\begin{proposition}\label{pro20}
Let $x^{\pm}=(x^{\pm}_1,\dots,x^{\pm}_N),~y=(y_1,\dots,y_N)$. Then
\begin{align}
\label{e:so8}\sum_{\mu=(\mu_1,\dots,\mu_N)}so_\mu(x^{\pm})s_\mu(y)=\frac{\prod^N_{k=1}(1+y_k)\prod_{1\leq k<l\leq N}(1-y_ky_l)}{\prod^{N}_{i,j=1}(1-x_iy_j)(1-x^{-1}_iy_j)}.
\end{align}
\end{proposition}
\begin{remark}
Relation \eqref{e:so8} is the classical Cauchy identity for odd orthogonal characters.
\end{remark}

Let us introduce the function $f_n(x)$ by its generating series
\begin{align}\label{e:r1}
\prod_{1\leq i\leq N}\frac{1}{(1-x_iz)(1-x_iz^{-1})}=\sum_{n\in \mathbb{Z}}f_n(x)z^n.
\end{align}
It is easy to see that
\begin{align}\label{e:r2}
f_n(x)=f_{-n}(x)=\sum^\infty_{i=0}h_i(x)h_{i+n}(x).
\end{align}
\begin{proposition}For generalized partitions $\mu=(\mu_1,\dots,\mu_l)$ and $\nu=(\nu_1,\dots,\nu_l)$, one has that
\begin{align}
\langle\mu^{so}|\Gamma_-(y)|\nu^{so}\rangle=SO^*_{\mu/\nu}(y),
\end{align}
where
\begin{align}\label{e:so21}
SO^*_{\mu/\nu}(y)=\det\left(f_{\nu_i-\mu_j-i+j}(y)-f_{\nu_i+\mu_j-i-j+2l+1}(y)\right)_{1\leq i,j\leq l}.
\end{align}
\end{proposition}
\begin{proof} It follows from definition that
\begin{align}\label{e:ve5}
\notag U^*(z)\Gamma_-(y)=&\prod^N_{i=1}\frac{1}{(1-y_iz)(1-y_iz^{-1})}\Gamma_-(y)U^*(z)\\
=&\sum_{i\in\mathbb{Z}}f_i(y)z^i\Gamma_-(y)U^*(z),
\end{align}
where we have used \eqref{e:r1} in the second equation. The equation \eqref{e:ve5} implies that
\begin{align*}
U^*_{-n}\Gamma_-(y)=\Gamma_-(y)\sum_{i\in\mathbb{Z}}f_i(y)U^*_{-n-i}.
\end{align*}
Thus by \eqref{e:so9} and \eqref{e:r2},
\begin{align*}
&\langle\mu^{so}|\Gamma_-(y)\\
=&\sum_{i_1,\dots,i_l\in\mathbb{Z}}f_{i_1}(y)\cdots f_{i_l}(y)\langle 0|U^*_{-\mu_l-i_l}\cdots U^*_{-\mu_1-i_1}\\
=&\sum_{\nu}\langle\nu^{so}|\sum_{\sigma\in S_l}\varepsilon(\sigma)\prod^l_{j=1}\left(f_{\nu_{\sigma(j)}-\mu_j-\sigma(j)+j}(y)-f_{\nu_{\sigma(j)}+\mu_j+2l+1-\sigma(j)-j}(y)\right)\\
=&\sum_{\nu}\langle\nu^{so}|\det\left(f_{\nu_i-\mu_j-i+j}(y)-f_{\nu_i+\mu_j-i-j+2l+1}(y)\right)_{1\leq i,j\leq l}.
\end{align*}
By the orthogonality relation \eqref{e:ortho1}, the result is proved.
\end{proof}
\begin{theorem}For generalized partition $\mu$ with length $l$ and $\lambda$ with length less than $l+N$, one has
\begin{align}
\label{e:so20}\sum_{\rho=(\rho_1,\dots,\rho_{l+N})}so_{\rho/\mu}(x^{\pm})SO^*_{\rho/\lambda}(y)=\prod^N_{i=1}\prod^K_{j=1}\frac{1}{(1-x_iy_j)(1-x^{-1}_iy_j)}\sum_{\tau=(\tau_1,\dots,\tau_l)}so_{\lambda/\tau}(x^{\pm})SO^*_{\mu/\tau}(y),
\end{align}
where $x^{\pm}=(x^{\pm}_1,\dots,x^{\pm}_N)$ and $y=(y_1,\dots,y_K)$ for any fixed $K\in\mathbb Z_+$.
\end{theorem}
\begin{proof}Using \eqref{e:identity}, for generalized partition $\nu=(\nu_1,\dots,\nu_n)$, we have
\begin{align}
\langle\nu^{so}|\sum_{\eta=(\eta_1,\dots,\eta_n)}|\eta^{so}\rangle\langle\eta^{so}|=\langle\nu^{so}|.
\end{align}
On the one hand,
\begin{align}\label{e:so23}
\notag\langle\mu^{so}|\Gamma_+(x^{\pm})\Gamma_-(y)|\lambda^{so}\rangle=&\langle\mu^{so}|\Gamma_+(x^{\pm})\sum_{\substack{\rho\\l(\rho)=l+N }}|\rho^{so}\rangle\langle\rho^{so}|\Gamma_-(y)|\lambda^{so}\rangle\\
=&\sum_{\rho=(\rho_1,\dots,\rho_{l+N})}so_{\rho/\mu}(x^{\pm})SO^*_{\rho/\lambda}(y),
\end{align}
where we have used \eqref{e:so21} and \eqref{e:so22}. On the other hand,
\begin{align}\label{e:so24}
\notag&\langle\mu^{so}|\Gamma_+(x^{\pm})\Gamma_-(y)|\lambda^{so}\rangle\\
\notag=&\prod^N_{i=1}\prod^K_{j=1}\frac{1}{(1-x_iy_j)(1-x^{-1}_iy_j)}\langle\mu^{so}|\Gamma_-(y)\sum_{\tau=(\tau_1,\dots,\tau_l)}|\tau^{so}\rangle\langle\tau^{so}|\Gamma_+(x^{\pm})|\lambda^{so}\rangle\\
=&\prod^N_{i=1}\prod^K_{j=1}\frac{1}{(1-x_iy_j)(1-x^{-1}_iy_j)}\sum_{\tau=(\tau_1,\dots,\tau_l)}so_{\lambda/\tau}(x^{\pm})SO^*_{\mu/\tau}(y).
\end{align}
Comparing \eqref{e:so23} with \eqref{e:so24}, we have derived \eqref{e:so20}.
\end{proof}
Similar to skew odd orthogonal characters, one also has the Cauchy-type identities for skew symplectic and orthogonal characters\footnote{There are slightly different Cauchy-type identities for skew symplectic/orthogonal characters in \cite[Theorem 5.8]{JLW2024}, where we have used different expressions to hand $\langle\mu^{sp}|\Gamma_-(\{y\})|\nu^{sp}\rangle$ and $\langle\mu^{o}|\Gamma_-(\{y\})|\nu^{o}\rangle$.}.
\begin{corollary}
Let $x^{\pm}=(x^{\pm}_1,\dots,x^{\pm}_N),~y=(y_1,\dots,y_K)$. For generalized partition $\mu$ with length $l$ and $\lambda$ with length less than $l+N$, then
\begin{align}
\label{e:sp16}&\sum_{\rho=(\rho_1,\dots,\rho_{l+N})}sp_{\rho/\mu}(x^{\pm})SP^*_{\rho/\lambda}(y)=\prod^N_{i=1}\prod^K_{j=1}\frac{1}{(1-x_iy_j)(1-x^{-1}_iy_j)}\sum_{\tau=(\tau_1,\dots,\tau_l)}sp_{\lambda/\tau}(x^{\pm})SP^*_{\mu/\tau}(y),\\
\label{e:o12}&\sum_{\rho=(\rho_1,\dots,\rho_{l+N})}o_{\rho/\mu}(x^{\pm})O^*_{\rho/\lambda}(y)=\prod^N_{i=1}\prod^K_{j=1}\frac{1}{(1-x_iy_j)(1-x^{-1}_iy_j)}\sum_{\tau=(\tau_1,\dots,\tau_l)}o_{\lambda/\tau}(x^{\pm})O^*_{\mu/\tau}(y),
\end{align}
where
\begin{align*}
&SP^*_{\alpha/\beta}(x)=\det\left(f_{\beta_i-\alpha_j-i+j}(x)-f_{\beta_i+\alpha_j-i-j+2(l+1)}(x)\right)_{1\leq i,j\leq l(\alpha)},\\
&O^*_{\alpha/\beta}(x)=\det\left(f_{\beta_i-\alpha_j-i+j}(x)+f_{\beta_i+\alpha_j-i-j+2l}(x)\right)_{1\leq i,j\leq l(\alpha)}.
\end{align*}
\end{corollary}
Setting $\lambda=\mu=0$, \eqref{e:so20} and $K=N$, \eqref{e:sp16} and \eqref{e:o12} are reduced to
\begin{align*}
\sum_{\rho=(\rho_1,\dots,\rho_N)} sp_{\rho}(x^{\pm})SP^*_{\rho}(y)&=\prod^N_{i,j=1}\frac{1}{(1-x_iy_j)(1-x^{-1}_iy_j)},\\
\sum_{\rho=(\rho_1,\dots,\rho_N)} so_{\rho}(x^{\pm})SO^*_{\rho}(y)&=\prod^N_{i,j=1}\frac{1}{(1-x_iy_j)(1-x^{-1}_iy_j)},\\
\sum_{\rho=(\rho_1,\dots,\rho_N)}\rho o_{\rho}(x^{\pm})O^*_{\rho}(y)&=\prod^N_{i,j=1}\frac{1}{(1-x_iy_j)(1-x^{-1}_iy_j)}.
\end{align*}

Recall the classical Cauchy identities \cite[Lemma 1.5.1]{KT1987} for symplectic characters, orthogonal characters and odd orthogonal characters
\begin{align*}
\sum_\rho sp_{\rho}(x^{\pm})s_{\rho}(y)&=\prod_{1\leq i<j\leq N}(1-y_iy_j)\prod^N_{i,j=1}\frac{1}{(1-x_iy_j)(1-x^{-1}_iy_j)},\\
\sum_\rho so_{\rho}(x^{\pm})s_{\rho}(y)&=\prod^N_{j=1}(1+y_j)\prod_{1\leq i<j\leq N}(1-y_iy_j)\prod^N_{i,j=1}\frac{1}{(1-x_iy_j)(1-x^{-1}_iy_j)},\\
\sum_\rho o_{\rho}(x^{\pm})s_{\rho}(y)&=\prod_{1\leq i\leq j\leq N}(1-y_iy_j)\prod^N_{i,j=1}\frac{1}{(1-x_iy_j)(1-x^{-1}_iy_j)},
\end{align*}
summed over all partitions $\rho$ with $l(\rho)\leq N$.
We then have the following identities related to Schur functions.
\begin{theorem}\label{e:th1}For partition $\lambda=(\lambda_1,\dots,\lambda_l)$, one has
\begin{align}
&\det\left(f_{\lambda_i-i+j}(x)-f_{\lambda_i-i-j+2l+2}(x)\right)_{1\leq i,j\leq l}= \frac{1}{\prod_{1\leq i<j\leq l}(1-x_ix_j)}s_\lambda(x),\\
\label{e:sum-pro1}&\det\left(f_{\lambda_i-i+j}(x)-f_{\lambda_i-i-j+2l+1}(x)\right)_{1\leq i,j\leq l}= \frac{1}{\prod_{1\leq j\leq l}(1+x_j)}\frac{1}{\prod_{1\leq i<j\leq l}(1-x_ix_j)}s_\lambda(x),\\
\label{e:sum-pro2}&\frac{1}{1+\delta_{\lambda_l, 0}}\det\left(f_{\lambda_i-i+j}(x)+f_{\lambda_i-i-j+2l}(x)\right)_{1\leq i,j\leq l}= \frac{1}{\prod_{1\leq i\leq j\leq l}(1-x_ix_j)}s_\lambda(x).
\end{align}
\end{theorem}
\begin{remark} Such determinants are also called Toeplitz or Hankel-type determinants \cite[Sec. 7]{Ste1990}.
\end{remark}
\begin{remark}
Let $\mathfrak{l}_\infty$ be
 the Levi subalgebra of type $A_{+\infty}$ obtained by removing a short (resp. long) simple root $\alpha_0$. Note that \eqref{e:r1} enjoys stability condition
 that $f_n(x_1, \ldots, x_N)|_{x_N=0}=f_n(x_1, \ldots, x_{N-1})$, so the limit $l\to\infty$ makes sense for the
  identities in Theorem \ref{e:th1}. When $l\to\infty$, the right side of \eqref{e:sum-pro1} or \eqref{e:sum-pro2} is the character of a generalized Verma module over $U_q(\mathfrak{x}_\infty)$ induced from a highest weight $\mathfrak{l}_\infty$-module with highest weight corresponding to $\lambda$, and
it is also the character of the set of pairs of semistandard Young tableaux $(S, T)$ of shape $\epsilon\tau$
and $\lambda$ with $\epsilon=1$ for $\mathfrak{l}=\mathfrak{b}$ and $\epsilon=2$ for $\mathfrak{l}=\mathfrak{c}$, respectively \cite[P.710, Remark 3.6]{Kwo2012}. Theorem \ref{e:th1} shows that these characters can be written as determinantal expressions.
\end{remark}
Let $\lambda=\varnothing$,
we have
\begin{corollary}\label{cor10}For $x=(x_1,\dots,x_l)$,
\begin{align*}
&\det\left(f_{-i+j}(x)-f_{-i-j+2l+2}(x)\right)_{1\leq i,j\leq l}=\frac{1}{\prod_{1\leq i<j\leq l}(1-x_ix_j)},\\
&\det\left(f_{-i+j}(x)-f_{-i-j+2l+1}(x)\right)_{1\leq i,j\leq l}=\frac{1}{\prod_{1\leq j\leq l}(1+x_j)}\frac{1}{\prod_{1\leq i<j\leq l}(1-x_ix_j)},\\
&\frac{1}{2}\det\left(f_{-i+j}(x)+f_{-i-j+2l}(x)\right)_{1\leq i,j\leq l}=\frac{1}{\prod_{1\leq i\leq j\leq l}(1-x_ix_j)}.
\end{align*}
\end{corollary}
\begin{remark} The RHS of the three identities in Corollary \ref{cor10} relate to the classical Littlewood identities
\begin{align*}
\sum_{\lambda^\prime \text{ even}} s_\lambda(x)
&= \prod_{1 \leq i<j \leq l}\frac{1}{1-x_ix_j},\\
\sum_{\lambda} s_\lambda(x)
&= \prod_{i=1}^n\frac{1}{1-x_i}\prod_{1\leq i<j \leq l}\frac{1}{1-x_ix_j},\\
\sum_{\lambda \text{ even}} s_\lambda(x)
&=\prod_{1\leq i\leq j \leq l}\frac{1}{1-x_ix_j},
\end{align*}
thus it is interesting to uncover the relation of Corollary \ref{cor10} with \cite[Theorem 7.1]{Ste1990}.
\end{remark}
\section{interpolating Schur polynomials and transitions for classical characters} 
Using vertex operators \eqref{e:so} and \eqref{e:ve1}, we can directly obtain transition laws between irreducible characters of $SO_{2N}, Sp_{2N}$ and $SO_{2N+1}$. We then construct three families of interpolating
Schur polynomials that interpolate among
characters of type $B$, $C$ and $D$. 

The following vertex operators were introduced in \cite{Ba1996,JN2015}
\begin{equation}
\begin{aligned}\label{e:ve1}
&Y(z)=\exp\left(\sum^\infty_{n=1}\frac{a_{-n}}{n}z^n\right)\exp\left(-\sum^\infty_{n=1}\frac{a_n}{n}(z^{-n}+z^n)\right)=\sum_{n\in \mathbb{Z}}Y_nz^{-n},\\
&Y^*(z)= (1-z^2)\exp\left(-\sum^\infty_{n=1}\frac{a_{-n}}{n}z^n\right)\exp\left(\sum^\infty_{n=1}\frac{a_n}{n}(z^{-n}+z^n)\right)=\sum_{n\in \mathbb{Z}}Y^*_nz^n,\\
&W(z)=(1-z^2)\exp\left(\sum^\infty_{n=1}\frac{a_{-n}}{n}z^n\right)\exp\left(-\sum^\infty_{n=1}\frac{a_n}{n}(z^{-n}+z^n)\right)=\sum_{n\in \mathbb{Z}}W_nz^{-n},\\
&W^*(z)=\exp\left(-\sum^\infty_{n=1}\frac{a_{-n}}{n}z^n\right)\exp\left(\sum^\infty_{n=1}\frac{a_n}{n}(z^{-n}+z^n)\right)=\sum_{n\in \mathbb{Z}}W^*_nz^n.
\end{aligned}
\end{equation}
For a given generalized partition $\lambda=(\lambda_1,\dots,\lambda_N)$, let
\begin{align*}
&|\lambda^{sp}\rangle=Y_{-\lambda_1}Y_{-\lambda_2}\cdots Y_{-\lambda_N}|0\rangle, \qquad\langle \lambda^{sp}|=\langle 0|Y^*_{-\lambda_N}\cdots Y^*_{-\lambda_1},\\
&|\lambda^{o}\rangle=W_{-\lambda_1}W_{-\lambda_2}\cdots W_{-\lambda_N}|0\rangle,\qquad\langle \lambda^{o}|=\langle 0|W^*_{-\lambda_N}\cdots W^*_{-\lambda_1}.
\end{align*}
\subsection{Transitions}
It follows from \cite[Proposition 3.3]{JLW2024} and \cite[Proposition 3.8]{JLW2024} that
\begin{equation}
 \begin{aligned}\label{e:tran2}
&\langle 0|\Gamma_+(x^{\pm})|\lambda^{sp}\rangle=sp_\lambda(x^{\pm}),\\
&\langle 0|\Gamma_+(x^{\pm})|\lambda^{o}\rangle=o_\lambda(x^{\pm}),
 \end{aligned}
 \end{equation}
 where $sp_\lambda(x^{\pm})$ (resp. $o_\lambda(x^{\pm})$) is the symplectic (resp. orthogonal) character.
 By \eqref{e:so} and \eqref{e:ve1}, we have
 \begin{align}\label{e:tran3}
 U_n=Y_n+Y_{n+1},~~~~W_n=U_n-U_{n+1},~~~~Y_n=\sum^\infty_{i=0}(-1)^iU_{n+i},~~~~U_n=\sum^\infty_{i=0}W_{n+i},
 \end{align}
 therefore
 \begin{equation}
 \begin{aligned}\label{e:tran1}
 &|\lambda^{so}\rangle=\sum_{\varepsilon\in \{0,1\}^N}Y_{-\lambda_1+\varepsilon_1}\cdots Y_{-\lambda_i+\varepsilon_i}\cdots Y_{-\lambda_N+\varepsilon_N}|0\rangle,\\
 & |\lambda^{o}\rangle=\sum_{\varepsilon\in \{0,1\}^N}(-1)^{|\varepsilon|}U_{-\lambda_1+\varepsilon_1}\cdots U_{-\lambda_i+\varepsilon_i}\cdots U_{-\lambda_N+\varepsilon_N}|0\rangle,
 \end{aligned}
 \end{equation}
 where $\varepsilon=(\varepsilon_1,\varepsilon_2,\dots,\varepsilon_N)$.
 Due to the fact that $Y_{m-1}Y_m=U_{m-1}U_m=0$ for any integer $m$, operators $Y_{-\lambda_1+\varepsilon_1}\cdots Y_{-\lambda_i+\varepsilon_i}\cdots Y_{-\lambda_N+\varepsilon_N}$ and $U_{-\lambda_1+\varepsilon_1}\cdots U_{-\lambda_i+\varepsilon_i}\cdots U_{-\lambda_N+\varepsilon_N}$ are nonzero unless $\lambda-\varepsilon=(\lambda_1-\varepsilon_1,\dots, \lambda_i-\varepsilon_i,\dots, \lambda_N-\varepsilon_N)$ is a generalized partition. Combining \eqref{e:so22}, \eqref{e:tran1} and \eqref{e:tran2}, we have the following transition formulas.
 \begin{theorem}\label{th11}For a partition $\lambda=(\lambda_1,\dots,\lambda_N)$,
 \begin{align}
& so_\lambda(x^{\pm})=\sum_{\mu=\lambda-\varepsilon}sp_{\mu}(x^{\pm}),\\
& o_\lambda(x^{\pm})=\sum_{\mu=\lambda-\varepsilon}(-1)^{|\lambda/\mu|}so_{\mu}(x^{\pm}),
 \end{align}
 where $\varepsilon=(\varepsilon_1,\dots, \varepsilon_N)$, $\varepsilon_i=0, 1$.
 \end{theorem}
 From \eqref{e:tran3}, we also have
 \begin{equation}
 \begin{aligned}
& |\lambda^{sp}\rangle=\sum_{n_i\geq 0}(-1)^{n_1+\dots+n_l}U_{-\lambda_1+n_1}\cdots U_{-\lambda_i+n_i}\cdots U_{-\lambda_N+n_N}|0\rangle,\\
 & |\lambda^{so}\rangle=\sum_{n_i\geq 0}W_{-\lambda_1+n_1}\cdots W_{-\lambda_i+n_i}\cdots W_{-\lambda_N+n_N}|0\rangle.
 \end{aligned}
 \end{equation}
 Similar to the process in deriving Theorem \ref{th11}, we have the following result.
 \begin{theorem}For any partition $\lambda$, one has that
 \begin{align*}
 &sp_\lambda(x^{\pm})=\sum_{\mu\prec\lambda}(-1)^{|\lambda/\mu|}so_\mu(x^{\pm}),\\
 &so_\lambda(x^{\pm})=\sum_{\mu\prec\lambda}o_\mu(x^{\pm}).
 \end{align*}
 \end{theorem}
 \subsection{$BD$-interpolating Schur polynomials $s^{BD}_\lambda(x;\alpha)$} For a generalized partition $\lambda=(\lambda_1,\dots,\lambda_N)$, define
 \begin{align}
 s^{BD}_\lambda(x;\alpha)=\langle 0|\bar{\Gamma}_+(x^{\pm};\alpha)|\lambda^{so}\rangle,
 \end{align}
 where $\bar{\Gamma}_+(x^{\pm};\alpha)$ is defined by \eqref{e:so32}. Using the BCH formula, one has
 \begin{align}
 \Gamma^{-1}_+(\alpha)U(z)=(1-\alpha z)U(z)\Gamma^{-1}_+(\alpha),
 \end{align}
 i.e.
  $\Gamma^{-1}_+(\alpha)U_n=(U_n-\alpha U_{n+1})\Gamma^{-1}_+(\alpha)$, therefore
 \begin{align}\label{e:so15}
  \Gamma^{-1}_+(\alpha)|\lambda^{so}\rangle=\sum_{\varepsilon\in \{0,1\}^N}(-\alpha)^{|\varepsilon|}U_{-\lambda_1+\varepsilon_1}\cdots U_{-\lambda_i+\varepsilon_i}\cdots U_{-\lambda_N+\varepsilon_N}|0\rangle.
 \end{align}
 For the case $\alpha=1$, one has $ \Gamma^{-1}_+(1)|\lambda^{so}\rangle=|\lambda^{o}\rangle$ by \eqref{e:tran1}. The case $\alpha=0$ force that $\varepsilon_i=0$, in other words, $ \Gamma^{-1}_+(0)|\lambda^{so}\rangle=|\lambda^{so}\rangle$. Therefore, the symmetric polynomials $s^{BD}_\lambda(x;\alpha)$ interpolate between characters of type $B$ and $D$ in the sense that
 \begin{align}
 s^{BD}_\lambda(x;0)=so_\lambda(x^{\pm}),~~s^{BD}_\lambda(x;1)=o_\lambda(x^{\pm}).
 \end{align}
 We now provide a determinantal expression for $s^{BD}_\lambda(x;\alpha)$.
 \begin{theorem}\label{BDth1} Let $\bar{h}_i(x^{\pm};\alpha)$ be the generalized homogeneous functions defined by $\frac{(1-\alpha z)}{\prod^N_i(1-x_iz)(1-x^{-1}_iz)}=\sum_{i\in \mathbb{Z}}\bar{h}_i(x^{\pm};\alpha)z^i$. Then for $\lambda=(\lambda_1,\dots,\lambda_N)$,
 \begin{eqnarray}
 s^{BD}_\lambda(x;\alpha)=\det\left(\bar{h}_{\lambda_i-i+j}(x^{\pm};\alpha)+\bar{h}_{\lambda_i-i-j+1}(x^{\pm};\alpha)\right)_{1\leq i,j\leq N}.
 \end{eqnarray}
 \end{theorem}
 \begin{proof}
 Using relation \eqref{e:so12a} and the orthonormal relation \eqref{e:ortho1}, one has
 \begin{align}
  \notag s^{BD}_\lambda(x;\alpha)&=\langle 0|\bar{\Gamma}_+(x^{\pm};\alpha)|\lambda^{so}\rangle\\
 \notag &=\sum_{\nu=(\nu_1,\dots,\nu_N)}so_\nu(x^{\pm})\langle \nu^{so}|\sum_{\varepsilon\in \{0,1\}^N}(-\alpha)^{|\varepsilon|}U_{-\lambda_1+\varepsilon_1}\cdots U_{-\lambda_i+\varepsilon_i}\cdots U_{-\lambda_N+\varepsilon_N}|0\rangle\\
 \notag &=\sum_{\varepsilon\in \{0,1\}^N}(-\alpha)^{|\varepsilon|}so_{\lambda-\varepsilon}(x^{\pm})\\
 \notag&=\det\left(h_{\lambda_i-i+j}(x^{\pm})+h_{\lambda_{i}-i-j+1}(x^{\pm})-\alpha(h_{\lambda_i-1-i+j}(x^{\pm})+h_{\lambda_{i}-i-j}(x^{\pm}))\right)_{1\leq i,j\leq N}\\
  &=\det\left(\bar{h}_{\lambda_i-i+j}(x^{\pm};\alpha)+\bar{h}_{\lambda_{i}-i-j+1}(x^{\pm};\alpha)\right)_{1\leq i,j\leq N},
 \end{align}
 where the last two equations have used the Jacobi--Trudi identity
 \begin{align*}
 so_\nu(x^{\pm})=\det\left(h_{\nu_i-i+j}(x^{\pm})+h_{\nu_{i}-i-j+1}(x^{\pm})\right)_{1\leq i,j\leq N}
 \end{align*}
 and the fact $\bar{h}_n(x^{\pm};\alpha)=h_n(x^{\pm})-\alpha h_{n-1}(x^{\pm})$.
 \end{proof}
\begin{remark}The proof of Theorem \ref{BDth1} also gives the transition formula between the $BD$-interpolating Schur polynomials and odd orthogonal characters
 \begin{align}
s^{BD}_\lambda(x;\alpha)=\sum_{\varepsilon\in \{0,1\}^N}(-\alpha)^{|\varepsilon|}so_{\lambda-\varepsilon}(x^{\pm}).
 \end{align}
 \end{remark}

 \subsection{$BC$-interpolating Schur polynomials $s^{BC}_\lambda(x;\alpha)$} For $\Gamma_+(\alpha)$ \eqref{e:eq10}, it is easy to check that (by vertex operator calculus)
  \begin{align}\label{e:BC1}
  \Gamma_+(0)|\lambda^{so}\rangle=|\lambda^{so}\rangle,~~~~~~~\Gamma_+(-1)|\lambda^{so}\rangle=|\lambda^{sp}\rangle,
  \end{align}
  where the second equation is due to \eqref{e:tran1}.

 Using \eqref{e:so31}, \eqref{e:tran2} and a similar method as in the proof of Theorem \ref{BDth1}, we get that
 \begin{align*}
 s^{BC}_\lambda(x;\alpha)&=\langle 0|\Gamma_+(x^{\pm};\alpha)|\lambda^{so}\rangle\\
 &=\det\left(h_{\lambda_i-i+j}(x;\alpha)+h_{\lambda_i-i-j+1}(x;\alpha)\right)^N_{ i,j=1},
 \end{align*}
 where $h_i(x^{\pm};\alpha)$ are defined by $\frac{1}{(1-\alpha z)\prod^N_i(1-x_iz)(1-x^{-1}_iz)}=\sum_{i\in \mathbb{Z}}h_i(x^{\pm};\alpha)z^i$.

Due to \eqref{e:BC1}, we call $s^{BC}_\lambda(x;\alpha)$ the {\it $BC$-interpolating Schur polynomials} in the sense that $ s^{BC}_\lambda(x;0)=so_\lambda(x^{\pm})$ and $s^{BC}_\lambda(x;-1)=sp_\lambda(x^{\pm})$.
 \subsection{$CD$-interpolating Schur polynomials $s^{CD}_\lambda(x;\alpha)$} Relations \eqref{e:ve1} tell us that
 \begin{align*}
 W_n=Y_n-Y_{n-2}.
 \end{align*}
 For a generalized partition $\lambda=(\lambda_1,\dots,\lambda_N)$, one has that
 \begin{align}\label{e:cd1}
 |\lambda^{o}\rangle=\sum_{\epsilon\in \{0,2\}^N}(-1)^{\sum^l_{i=1}\delta_{\epsilon_i\neq 0}}Y_{-\lambda_1+\epsilon_1}\cdots Y_{-\lambda_i+\epsilon_i}\cdots Y_{-\lambda_N+\epsilon_N}|0\rangle,
 \end{align}
 where $\epsilon=(\epsilon_1,\epsilon_2\dots,\epsilon_N)$.
 From the definition of $\widetilde{\Gamma}_+(\alpha)$ \eqref{e:cd2}, we have
 \begin{align}\label{e:cd3}
 \widetilde{\Gamma}_+(0) |\lambda^{sp}\rangle=|\lambda^{sp}\rangle,~~~~~~ \widetilde{\Gamma}_+(1) |\lambda^{sp}\rangle=|\lambda^{o}\rangle.
 \end{align}

Using \eqref{e:so30}, \eqref{e:tran2} and the proof of Theorem \ref{BDth1}, we have
 \begin{align*}
 s^{CD}_\lambda(x;\alpha)&=\langle 0|\widetilde{\Gamma}_+(x^{\pm};\alpha)|\lambda^{sp}\rangle\\
 &=\det\left(\tilde{h}_{\lambda_i-i+j}(x^{\pm};\alpha)+\delta_{j>1}\tilde{h}_{\lambda_i-i-j+2}(x^{\pm};\alpha)\right)^N_{ i,j=1},
 \end{align*}
 where $\tilde{h}_i(x^{\pm};\alpha)$ satisfy $\frac{(1-\alpha z^2)}{\prod^N_i(1-x_iz)(1-x^{-1}_iz)}=\sum_{i\in \mathbb{Z}}\tilde{h}_i(x^{\pm};\alpha)z^i$.

 Equations \eqref{e:cd3} show that $ s^{CD}_\lambda(x;0)=sp_\lambda(x^{\pm})$ and $s^{CD}_\lambda(x;1)=o_\lambda(x^{\pm})$, i.e., symmetric polynomials $s^{CD}_\lambda(x;\alpha)$ interpolate between characters of type $C$ and $D$. We call $s^{CD}_\lambda(x;\alpha)$ the {\it $CD$-interpolating Schur polynomials}.
 
\appendix
 \section{dual Jacobi-Trudi identity}
 \begin{center}
\sc by Naihuan Jing, Zhijun Li, Xinyu Pan,Danxia Wang and Chang Ye
\end{center}
 \setcounter{equation}{0}
\renewcommand\theequation{A\arabic{equation}}

Abion et al's \cite{AFHS2023} have recently proved the dual Jacobi-Trudi identity for skew odd orthogonal characters $so_{\lambda/\mu}(x^{\pm})$
by the lattice path method. In the above we have used vertex algebraic method to realize the skew odd orthogonal characters $so_{\lambda/\mu}(x^{\pm})$ and derive
various combinatorial structures including the Jacobi-Trudi formula. In this appendix, we give a
vertex operator approach for the dual version as well.

 For partition $\nu=(\nu_1,\dots,\nu_l)$, we introduce the following partition element of $\mathcal{M}$
\begin{align}
|\underline{\nu^{so}}\rangle=U^*_{\nu_1}\cdots U^*_{\nu_l}|0\rangle.
\end{align}
\begin{proposition}\label{pro7}For the generalized partitions $\mu=(\mu_1,\mu_2,\dots,\mu_l)$ and $\lambda=(\lambda_1,\lambda_2,\dots,\lambda_k)$ with $\lambda_1\leq l$, one has that
\begin{align}\label{e:ortho2}
\langle \mu^{so}|\underline{\lambda^{so}}\rangle=(-1)^{|\mu|}\delta_{\lambda^{\prime}\mu},
\end{align}
where some parts of $\lambda$ and $\mu$ can be zeros.
\end{proposition}
\begin{proof} It follows from \eqref{e:ortho1} that
\begin{align*}
\langle \mu^{so}|\underline{\lambda^{so}}\rangle=\langle 0|U^*_{-\mu_l}\cdots U^*_{-\mu_1}U^*_{\lambda_1}\cdots U^*_{\lambda_k}|0\rangle
\end{align*}
is nonzero unless there exists a permutation $\sigma\in S_{l+k}$ such that
\begin{align}
\langle 0|U^*_{-\mu_l}\cdots U^*_{-\mu_1}U^*_{\lambda_l}\cdots U^*_{\lambda_k}=\varepsilon(\sigma)\langle 0|\underbrace{U^*_0\cdots U^*_0}_{l+k}.
\end{align}
Using the relation \eqref{e:so9}, we have
\begin{equation}
\begin{aligned}\label{e:aa1}
&\sigma(i)-i=-\mu_{i-k}, ~~~~~~~~~~k+1\leq i\leq k+l\\
&\sigma(i)-i=\lambda_{k+1-i},~~~~~~~~~1\leq i\leq k,
\end{aligned}
\end{equation}
where we used the fact $2(l+k)+1-\sigma(i)-i> l$ for $1\leq i\leq k$. In fact, \eqref{e:aa1} is equivalent to the bijection \cite[(2.22)]{JL2022}
\begin{align}
\{k+l-\mu_l,\dots,k+1-\mu_1;\lambda_1+k,\dots,\lambda_k+1\}\stackrel{\sigma}{\leftrightarrow} \{1,\dots,k+l\}.
\end{align}
From the proof of \cite[Proposition 2.3]{JL2022}, we know that $\lambda^{\prime}=\mu$ and $\varepsilon(\sigma)=(-1)^{|\mu|}$, which have completed the proof.
\end{proof}
\begin{proposition}For $x^{\pm}=(x^{\pm}_1,\dots,x^{\pm}_N)$ and partition $\eta$ with $\eta_1\leq N$, one has that
\begin{align}
\langle 0|\Gamma_+(x^{\pm})|\underline{\eta^{so}}\rangle=(-1)^{|\eta|}so_{\eta^{\prime}}(x^{\pm}).
\end{align}
\end{proposition}
\begin{proof}By \eqref{e:so12a} and \eqref{e:ortho2}, we have
\begin{align*}
\langle 0|\Gamma_+(x^{\pm})|\underline{\eta^{so}}\rangle=&\sum_{\nu}so_\nu(x^{\pm})\langle \nu^{so}||\underline{\eta^{so}}\rangle\\
=&(-1)^{|\eta|}so_{\eta^{\prime}}(x^{\pm}).
\end{align*}
\end{proof}
For $l(\lambda)\leq l+N$ and $y^{\pm}=(y^{\pm}_1,\dots,y^{\pm}_l),~x^{\pm}=(x^{\pm}_1,\dots,x^{\pm}_N)$, we have
\begin{equation}
\begin{aligned}
so_\lambda(y^{\pm};x^{\pm})=&(-1)^{|\lambda|}\langle 0|\Gamma_+(y^{\pm};x^{\pm})|\underline{\lambda^{'so}}\rangle\\
=&(-1)^{|\lambda|}\langle 0|\Gamma_+(y^{\pm})\sum_{\mu=(\mu_1,\dots,\mu_l)}|\mu^{so}\rangle\langle \mu^{so}|\Gamma_+(x^{\pm})|\underline{\lambda^{'so}}\rangle\\
=&\sum_{\mu=(\mu_1,\dots,\mu_l)}so_\mu(y^{\pm})(-1)^{|\lambda|}\langle \mu^{so}|\Gamma_+(x^{\pm})|\underline{\lambda^{'so}}\rangle.
\end{aligned}
\end{equation}
In other words, $(-1)^{|\lambda|}\langle \mu^{so}|\Gamma_+(x^{\pm})|\underline{\lambda^{'so}}\rangle$ is another vertex operator realization of skew odd orthogonal character $so_{\lambda/\mu}(x^{\pm})$.
\begin{theorem}For partitions $\mu=(\mu_1,\dots,\mu_l)$ and $\lambda$ with $l(\lambda)\leq l+N$ and $\lambda_1\leq s$, one has
\begin{align}
so_{\lambda/\mu}(x^{\pm})=\det\left(e_{\lambda^{\prime}_i-\mu^{\prime}_j-i+j}(x^{\pm})+e_{\lambda^{\prime}_i+\mu^{\prime}_j-i-j-2l+1}(x^{\pm})\right)^{s}_{i,j=1}.
\end{align}
\end{theorem}
\begin{proof} Recalling the definition of vertex operator $U^*(z)$ and $\Gamma_+(x^{\pm})$, we have that
\begin{align*}
\Gamma_+(x^{\pm})U^*(z)=\prod^N_{i=1}(1-x_iz)(1-x^{-1}_iz)U^*(z)\Gamma_+(x^{\pm}),
\end{align*}
i.e., $
\Gamma_+(x^{\pm})U^*_n=\sum^\infty_{i=0}(-1)^ie_i(x^{\pm})U^*_{n-i}\Gamma_+(x^{\pm})$,
where $e_i(x^{\pm})$ is the elementary symmetric function defined by the generating series
 $\prod^N_{i=1}(1+x_iz)(1+x^{-1}_iz)=\sum^\infty_{i=0}e_i(x^{\pm})z^i$. We thus have
\begin{align}
\notag\langle \mu^{so}|\Gamma_+(x^{\pm})|\underline{\lambda^{'so}}\rangle=&\langle \mu^{so}|\Gamma_+(x^{\pm})U^*_{\lambda^{\prime}_1}\cdots U^*_{\lambda^{\prime}_s}|0\rangle\\
=&\sum_{i_1,\dots,i_s\geq 0}(-1)^{i_1+\dots+i_s}e_{i_1}(x^{\pm})\cdots e_{i_s}(x^{\pm})\langle 0|U^*_{-\mu_l}\cdots U^*_{-\mu_1}U^*_{\lambda^{\prime}_1-i_1}\cdots U^*_{\lambda^{\prime}_s-i_s}|0\rangle,
\end{align}
here we have used the fact $\Gamma_+(x^{\pm})|0\rangle=|0\rangle$. The orthogonality relation \eqref{e:ortho1} tells us that
$$\langle 0|U^*_{-\mu_l}\cdots U^*_{-\mu_1}U^*_{\lambda^{\prime}_1-i_1}\cdots U^*_{\lambda^{\prime}_s-i_s}|0\rangle$$ is zero unless 
\begin{align*}
\langle 0|U^*_{-\mu_l}\cdots U^*_{-\mu_1}U^*_{\lambda^{\prime}_1-i_1}\cdots U^*_{\lambda^{\prime}_s-i_s}=\epsilon\langle 0|\underbrace{U^*_0\cdots U^*_0}_{l+s}
\end{align*}
up to sign.
Then by \eqref{e:so9}, we have
\begin{equation}\label{e:A1}
\left\{\begin{aligned}
&-\mu_i=\sigma(s+i)-(s+i)~~&~~1\leq i\leq l,\\
&\lambda^{\prime}_j-i_j=\sigma(s+1-j)-(s+1-j)~\&~-\sigma(s+1-j)+2l+s+j~~~&~~1\leq j\leq s.
\end{aligned}
\right.
\end{equation}
\eqref{e:A1} is equivalent to
\begin{equation}\label{e:A2}
\left\{\begin{aligned}
&l+s+1-\sigma(s+i)=\mu_i+l-i+1~~&~~1\leq i\leq l,\\
&l+s+1-\sigma(s+1-j)=-\lambda^{\prime}_j+i_j+l+j~\&~\lambda^{\prime}_j-i_j-l-1-j~~~&~~1\leq j\leq s,
\end{aligned}
\right.
\end{equation}
where the left of \eqref{e:A2} is a permutation of $\{1,\dots,l+s\}$. By \cite[(1.7)]{Mac1995},
\begin{align*}
\mu_i+l-i+1 ~~(1\leq i\leq l),~~~~~~~~l-1+k-\mu^{\prime}_k~~(1\leq k\leq s)
\end{align*}
is also a permutation of $\{1,\dots,l+s\}$, so there exists $\tau\in S_s$ such that
\begin{align*}
l-1+\tau(j)-\mu^{\prime}_{\tau(j)}=-\lambda^{\prime}_j+i_j+l+j~\&~\lambda^{\prime}_j-i_j-l-1-j~~~&~~1\leq j\leq s,
\end{align*}
i.e.,
\begin{align*}
i_j=\lambda^{\prime}_j-\mu^{\prime}_{\tau(j)}+\tau(j)-j~\&~\lambda^{\prime}_j+\mu^{\prime}_{\tau(j)}-\tau(j)-j-2l+1~~~&~~1\leq j\leq s.
\end{align*}
In addition, $\epsilon=(-1)^{|\mu|+\tau}$ from the proof of Proposition \ref{pro7}. Then using \eqref{e:so9} we have proved the theorem.
\end{proof}

\end{document}